\documentclass[11pt,a4paper]{article}
\usepackage{geometry} 
\usepackage{amsthm,amsmath,amsfonts,amscd,latexsym,amssymb,amsbsy,dsfont,mathrsfs,tabls}

\newtheorem{theorem}{Theorem}[section]
\newtheorem{lemma}[theorem]{Lemma}

\newtheorem{corollary}[theorem]{Corollary}

\newtheorem{definition}[theorem]{Definition}

\newtheorem{remark}[theorem]{Remark}

\newcommand{\C}{\mbb{C}}
\newcommand{\HH}{\mbb{H}}

\newcommand{\I}{\mscr{I}}
\newcommand{\II}{\mc{I}}

\newcommand{\R}{\mbb{R}}
\newcommand{\Q}{\mc{Q}}
\newcommand{\B}{\mc{B}}

\newcommand{\OO}{\Omega}

\newcommand{\mb}{\mathbf}
\newcommand{\mbb}{\mathbb}
\newcommand{\mc}{\mathcal}

\newcommand{\mr}{\mathrm}

\newcommand{\mscr}{\mathscr}
\newcommand{\lra}{\longrightarrow}
\newcommand{\pr}{\prime}

\newcommand{\ui}{i}

\def\SS{\mathbb S}
\newcommand\Ac{A_{\C}}

\newcommand\Sl{\mathcal S}

\newcommand\dd[2]{\dfrac{\partial#1}{\partial#2}}

\newcommand{\cS}{\mbb{S}}

\newcommand{\wrt}{w.r.t.\ }
\newcommand{\n}{\mr{n}}
\newcommand{\nn}{\mb{n}}

\newcommand{\im}{\mr{Im}}
\newcommand{\dck}{\Gamma_S}
\newcommand{\ck}{C_S}

\newcommand{\be}{\begin{equation}}
\newcommand{\ee}{\end{equation}}

\begin{document}

\title{{\itshape Volume Cauchy formulas for slice functions\\ on real associative $^*$-algebras}}

\author{R.\ Ghiloni, A.\ Perotti
\thanks{Work partially supported by GNSAGA of INdAM}
\\Department of Mathematics\\
  University of Trento\\ Via Sommarive, 14\\ I--38123 Povo Trento ITALY\\
ghiloni@science.unitn.it,   perotti@science.unitn.it}

\date{}

\maketitle

\begin{abstract}
We introduce a family of Cauchy integral formulas for slice and slice regular functions on a real associative *-algebra. For every  suitable choice of a real subspace of the algebra, a different formula is given, in which the domains of integration are subsets of the subspace. In particular, in the quaternionic case we get a volume Cauchy formula. In the Clifford algebra case, the choice of the paravector subspace $\R^{n+1}$ gives a volume Cauchy formula for slice monogenic functions.
\medskip


\end{abstract}


\section{Introduction}

The concept of slice regularity for functions of one quaternionic, octonionic or Clifford variable has been introduced recently by Gentili and Struppa \cite{GeSt2006CR,GeSt2007Adv} and by Colombo, Sabadini and Struppa \cite{CoSaSt2009Israel}.
In \cite{GhPe_AIM} and \cite{GhPe_Trends}, a new approach to slice functions, based  on the concept of \emph{stem function}, allowed to extend further the theory to any real alternative $^*$-algebra of finite dimension. 
In this setting, a Cauchy integral formula for slice functions of class $\mscr{C}^1$ was proved. In the  quaternionic case, a Cauchy kernel was already introduced in \cite{CoGeSaPreprint2008}, and for slice monogenic functions in \cite{CoSaJMAA2011,CoSaTM2009}.  This kernel was applied in \cite{CoSaStMMJ2011} to get Cauchy formulas for $\mscr{C}^1$-functions on a class of domains intersecting the real axis. 

In the present work, we introduce a family of Cauchy integral formulas for slice and slice regular functions, in which  the domains of integration depend on the choice of a suitable real vector subspace of the algebra. In the quaternionic case, taking as subspace the whole space $\HH$, we get a \emph{volume Cauchy formula}. In the Clifford algebra case, we can choose the paravector subspace $\R^{n+1}$, obtaining a Cauchy formula in which integration is performed on an open subset of $\R^{n+1}$ and on its boundary.

We begin by fixing some assumptions and by recalling some basic notions. 

\textit{Fix a real associative algebra $A$  with unity $1$ of finite dimension $d>0$, equipped with the $\mscr{C}^{\infty}$-manifold structure as a real vector space and with an anti-involution $x \longmapsto x^c$. Define $e:=d-1$}. Identify $\R$ with the subalgebra of $A$ generated by~$1$. 
The anti-involution $x \longmapsto x^c$ is a real linear map of $A$ into $A$ satisfying the following properties: $(x^c)^c=x$ for all $x \in A$, $(xy)^c=y^cx^c$ for all $x,y\in A$ and $x^c=x$ for each real $x$. We can then consider $A$ as a real $^*$-algebra.

For each element $x$ of $A$, the \emph{trace} of $x$ is $t(x):=x+x^c\in A$ and the (squared) \emph{norm} of $x$ is
$n(x):=xx^c\in A$. We recall some definitions from \cite{GhPe_AIM} and \cite{GhPe_Trends}.

\begin{definition}\label{cone}
The \emph{quadratic cone} of  $A$ is the set
\[\Q_A:=\R\cup\{x\in A\ |\ t(x)\in\R,\ n(x)\in\R,\ 4n(x)> t(x)^2\}.\]
We also set\,
$\SS_A:=\{J\in \Q_A\ |\ J^2=-1\}$. Elements of\, $\SS_A$ are called \emph{square roots of $-1$} in the algebra $A$. For each $J\in \SS_A$, we will denote by $\C_J:=\langle 1,J\rangle\simeq\C$ the subalgebra of $A$ generated by  $J$. 
\end{definition}
\textit{In what follows, we assume that $\cS_A \neq \emptyset$.} It follows that $d$ is even.

Let $\Ac=A\otimes_{\R}\C$ be the complexification of $A$. We will use the representation $\Ac=\{w=x+i y \ | \ x,y\in A\}$, with  $i^2=-1$ and complex conjugation $\overline w=\overline{x+iy}=x-iy$.

Let $D$ be a non-empty subset of $\C$, invariant under the complex conjugation $z=\alpha+\ui \beta \longmapsto \overline{z}=\alpha-\ui \beta$. A function $F:D \lra \Ac$ is called a \textit{stem function} on $D$ if it satisfies the condition $F(\overline z)=\overline{F(z)}$ for each $z \in D$. If $F_1,F_2:D \lra A$ are the $A$-valued components of $F=F_1+iF_2$, then such a condition is equivalent to require that $F_1(\overline z)=F_1(z)$ and $F_2(\overline z)=-F_2(z)$ for each $z \in D$. We call $F$ \textit{continuous} if $F_1$ and $F_2$ are continuous. We say that $F$ is \textit{of class $\mscr{C}^1$} if $F_1$ and $F_2$ can be extended on an open neighborhood $U$ of $D$ in $\C$ to functions $\tilde{F}_1,\tilde{F}_2$ of class $\mscr{C}^1$ in the usual sense. The reader observes that one can also suppose that $U$ is invariant under the complex conjugation of $\C$ and $\tilde{F}_1+i\tilde{F}_2$ is a stem function.

Let $\OO_D$ be the  subset of $\Q_A$ defined by:
\[
\OO_D:=\{x \in \Q_A \,|\, x=\alpha+\beta J, \ \alpha,\beta \in \R, \ \alpha+i\beta \in D, \ J \in \SS_A\}.
\]
We set $D_J:=\OO_D\cap\C_J$ and denote by $\partial D_J$ the (relative topological) boundary of $D_J$ in $\C_J$. Observe that, if $D$ is open, then $\OO_D$ is relatively open in $\Q_A$ and the boundary $\partial \OO_D$ of $\OO_D$ in $\Q_A$ coincides with $\bigcup_{J \in \SS_A}\partial D_J$. 

\begin{definition}
Any stem function $F=F_1+iF_2:D \lra \Ac$ induces a \emph{$($left$)$ slice function} $f=\II(F):\OO_D \lra A$ as follows: if $x=\alpha+\beta J\in D_J$ for some $\alpha,\beta \in \R$ and $J \in \SS_A$, we set  
\[
f(x):=F_1(z)+JF_2(z) \quad (z=\alpha+i\beta).
\]
\end{definition}

We will denote by $\Sl^0(\OO_D,A)$ the real vector space of (left) slice functions on $\OO_D$ induced by continuous stem functions and  by
$\Sl^1(\OO_D,A)$ 
the real vector space of slice functions induced by stem fun\-ctions of class $\mscr{C}^1$. Proposition 7 of \cite{GhPe_AIM} ensures that $\mc{S}^h(\OO_D,A) \subset \mscr{C}^0(\OO_D,A)$ ($h=0,1$).

Suppose that $D$ is a non--empty open subset of $\C$. Let $F:D \lra \Ac$ be a stem function of class $\mscr{C}^1$ and let $f=\II(F) \in \mc{S}^1(\OO_D,A)$. Let us denote by $\partial F/\partial \overline{z}:D \lra \Ac$ the stem function on $D$ defined by
\[
\dd{F}{\overline{z}}:=\frac{1}{2}\left(\dd{F}{\alpha}+\ui \dd{F}{\beta} \right),
\]
which induces the slice derivative $\dd{f}{x^c}:=\II\left(\dd{F}{\overline{z}}\right)\in \mc{S}^0(\OO_D,A)$.

\begin{definition}
A slice function $f\in\mc{S}^1(\OO_D,A)$ is called \emph{slice regular} if it holds:
\[
\dd{f}{x^c}=0 \text{\quad on }\OO_D.
\]
We denote by $\mc{SR}(\OO_D,A)$ the real vector space of all slice regular functions on~$\OO_D$.
\end{definition}


Let $S$ be a non--empty subset of $A$. We say that $S$ is a \emph{genuine imaginary sphere of $A$}, for short a \emph{gis of $A$}, if there exists a (real) vector subspace $M$ of $A$ such that $\R \subset M \subset \Q_A$ and $S=M \cap \cS_A$. If such a $M$ exists, then it is unique. In fact, it is easy to verify that $M=\bigcup_{J \in S}\C_J$. For this reason, if $S$ is a gis of $A$, then we say that $M$ is the \textit{vector subspace of $A$ inducing $S$}. The reader observes that, since $M$ contains $1$ and at least one element of $\cS_A$, its dimension is at least 2.  Moreover, the set $\{J,-J\}$ is a gis of $A$ for each $J \in \cS_A$.

\begin{lemma} \label{lem:gis}
Let $S$ be a gis of $A$ and let $M$ be the vector subspace of $A$ inducing $S$. Then there exists a norm $\|\ \|$ on $A$ such that $\|x\|^2=n(x)$ for each $x \in M$.
\end{lemma}
\begin{proof}
Since $M \subset \Q_A$, the function $n(x)$ is real-valued and non-negative on $M$. We prove that the function $\sqrt{n(x)}$ is a norm on $M$. First of all, the function is positive-homogeneous and vanishes only at $x=0$. It remains to prove that it satisfies the triangle inequality:
\[
\sqrt{n(x+y)}\le\sqrt{n(x)}+\sqrt{n(y)}\quad\text{for each }x,y\in M.
\]
This is equivalent to 
\be \label{norm2}
n(x+y)-n(x)-n(y)\le 2\sqrt{n(x)n(y)}\;.
\ee
The left-hand side of \eqref{norm2} is $xy^c+yx^c=t(xy^c)$.
Let $x=\alpha+\beta I$, $y=\alpha'+\beta'J$ in $M$, with $I,J\in S\subset\cS_A$, $\alpha,\beta,\alpha',\beta'\in\R$ and $\beta,\beta' \geq 0$. Since $I- J\in M$, we have that $n(I-J) \in \R$ and
\[
0\le n(I-J)=-(I- J)^2=2+ t(IJ).
\]
It follows that $t(IJ)$ is real and $t(IJ)\ge-2$. Since $xy^c=(\alpha+\beta I)(\alpha'-\beta' J)$, the trace of $xy^c$ is equal to 
$2\alpha\alpha'-\beta\beta' t(IJ)$. Therefore, $t(xy^c)\le 2(\alpha\alpha'+\beta\beta')$, while $n(x)n(y)=(\alpha^2+\beta^2)(\alpha'^2+\beta'^2)\ge(\alpha\alpha'+\beta\beta')^2$. It follows that  $t(xy^c)\le2\sqrt{n(x)n(y)}$, which is precisely inequality \eqref{norm2}. 
Since the restriction of $\sqrt{n(x)}$ to $M$ is a norm, it can be extended to a norm $\|\ \|$ on $A$.
\end{proof}


\section{The volume Cauchy formulas}

\textit{Fix a gis $S$ of $A$. Denote by $M$ the vector subspace of $A$ inducing $S$ and by $m$ the dimension of $M$. Choose a norm $\|\ \|$ on $A$ as in Lemma \ref{lem:gis}}.

Let $\B=(v_0,v_1,\ldots,v_e)$ be a (real) vector basis  of $A$ with $v_0=1$, orthonormal w.r.t.\ the scalar product $(\ ,\,)$ on $A$ associated to the norm $\|\ \|$ and such that  $v=(v_0,v_1,\ldots,v_{m-1})$ form a orthonormal basis of $M$. Note that $M$ is the orthogonal direct sum of $\R$ and $M \cap \ker(t)$, since $(x,y)=\frac12 t(xy^c)$ on $M$. Let $L:\R^d \lra A$ be the real vector isomorphism sending $x=(x_0,x_1,\ldots,x_e)$ into $L(x)=\sum_{\ell=0}^ex_{\ell}v_{\ell}$. Identify $\R^d$ with $A$ via~$L$ and hence $M$ with $\R^m=\R^m \times \{0\} \subset \R^m \times \R^{d-m}=\R^d$. The product of $A$ becomes a product on $\R^d$ by requiring that $L$ is an isomorphism of $\R$-algebras. In other words, given $x,y \in \R^d$, $xy$ is defined as $L^{-1}(L(x)L(y))$. Since $\B$ is orthonormal, $\|x\|$ coincides with the usual euclidean norm $(\sum_{\ell=0}^ex_{\ell}^2)^{1/2}$ of $x$ in $\R^d$. By Proposition 1(6) of \cite{GhPe_AIM}, we know that $\cS_A=\{J \in A \, | \, t(J)=0, \ n(J)=1\}$. In this way, we have that
\[
\textstyle
S=\{(x_0,x_1,\ldots,x_{m-1}) \in \R^m \, | \, x_0=0, \, \sum_{\ell=1}^{m-1}x_{\ell}^2=1\}.
\]

\textit{In what follows, we assume that $D$ is a non--empty bounded open subset of $\C$ with boundary $\partial D$  of class $\mscr{C}^1$}. Denote by $\n:\partial D \lra \C$ the continuous function sending $z \in \partial D$ into the outer normal versor of $\partial D$ at $z$.  
Since $D$ is stable under conjugation, the map $\n$ is a 
 stem function on $\partial D$. 

We define the \emph{circularization $\OO_D(S)$ of $D$ \wrt $S$} as the following subset of $\Q_A$:
\[
\OO_D(S):=\{x \in \Q_A \, | \, x=\alpha+\beta J, \ \alpha,\beta \in \R, \ \alpha+i\beta \in D, \ J \in S\}.
\]
Since $\OO_D(\{J,-J\})=D_J$ for each $J \in \cS_A$, it follows that  $\OO_D(S)=\bigcup_{J \in S}D_J$. On the other hand, $M=\bigcup_{J \in S}\C_J$ and hence $\OO_D(S)$ is an open subset of~$M$. Denote by $\partial \OO_D(S)$ the boundary of $\OO_D(S)$ in $M$. It is easy to see that $\partial \OO_D(S)=\bigcup_{J \in S}\partial D_J=\OO_{\partial D}(S)$. We define the \emph{outer normal vector field $\nn:\partial\OO_D(S) \lra A$ to $\partial\OO_D(S)$} as the slice function induced by the stem function $\n$. More explicitly, $\nn$ is defined as follows. Given $x=\alpha+\beta J \in \partial\OO_D(S)$, $z:=\alpha+i\beta$ belongs to $\partial D$ and hence we can write $\n(z)=\n_1+i\n_2$ for some $\n_1,\n_2 \in \R$. Then we have: $\nn(x):=\n_1+\n_2 J$.


For each non-negative integer $n$, we denote by $\eta_{n}$ the volume of the standard sphere $S^{n}$ of $\R^{n+1}$. It is well known that $\eta_n$ has the following explicit expression:
\[
\eta_{n}=\frac{2\pi^{\frac{n+1}2}}{\Gamma\left(\frac{n+1}2\right)}\;, 
\]
where $\Gamma$ is Euler's gamma function.

Let us introduce a notion of Cauchy kernel of $A$ relative to the gis  $S$.

We start recalling from \cite{GhPe_AIM} the definitions of the characteristic polynomial $\Delta_w$ of $w \in \Q_A$ and of the Cauchy kernel of $A$. $\Delta_w$ is the slice regular polynomial
\[
\Delta_{w}(x):=x^2-x\,t(w)+n(w),
\]
with zero set $\cS_w:=\{x\in \Q_A\;|\; t(x)=t(w),\ n(x)=n(w)\}$.
The Cauchy kernel for slice regular functions on $A$ is defined, for each $x\in \Q_A\setminus \SS_w$, as 
\[
C(x,w):=\Delta_w(x)^{-1}(w^c-x).
\]
$C(\,\cdot\,,w)$ is slice regular on $\Q_A\setminus \SS_w$ and has the following property \wrt the slice product of functions:
\[
C(x,w)\cdot (w-x)=1.
\]

Define $\dck:=\{(x,w) \in \Q_A \times (M \setminus \R) \, | \, \Delta_w(x) \neq 0\}$. Observe that, if $(x,w) \in \dck$ and $x=\alpha+\beta J$ for some $\alpha,\beta \in \R$ and $J \in S$, then $\Delta_w(x)$ belongs to $\C_J \setminus \{0\}$ and hence $\Delta_w(x)$ is invertible in $A$, or better, in $\C_J$. This fact ensures that the following definition is consistent.

\begin{definition}
We define the \emph{Cauchy kernel of $A$ \wrt  $S$} as the smooth function $\ck:\dck \lra A$ given by setting
\[
\ck(x,w):=\frac{2}{\eta_{m-2}} \, \frac{C(x,w)}{ \big(n(\im(w))\big)^{\frac{m-2}{2}}}=\frac{2}{\eta_{m-2}} \, \frac{C(x,w)}{\|\im(w)\|^{m-2}} \, ,
\]
where $\im(w):=(w-w^c)/2$.
\end{definition}
Note that, for each fixed $w \in M \setminus \R$, $\ck(\,\cdot\,,w)$ is a slice regular function on $\Q_A\setminus\cS_w$. In particular, when $S=\{J,-J\}$  for some $J \in \cS_A$, then $m=2$, $\eta_{m-2}=2$, and $C_S$ extends up to the real axis and coincides with the kernel $C$ of $A$.

Denote by $w=(w_0,w_1,\ldots,w_{m-1}):M \lra \R^m$ the coordinate system on $M$ sending $w$ into $\sum_{\ell=0}^{m-1}w_{\ell}v_{\ell}$ and by $dw$ the corresponding volume form $dw_0 \wedge dw_1 \wedge \cdots \wedge  dw_{m-1}$. Such a volume form induces a structure of oriented Riemannian manifold on $M$ and hence on its open subset $\OO_D(S)$. Its boundary $\partial \OO_D(S)$ is a hypersurface of $M$ of class $\mscr{C}^1$, which inherits a structure of oriented Riemannian manifold, via the standard rule ``first the outer normal vector''. Denote by $d\sigma_w$ the corresponding volume form of $\partial \OO_D(S)$. By abusing notation, we will use the letter $w$ to indicate both a point of $M$ and its coordinates \wrt $v$.

We have:

\begin{lemma} \label{lem:L^1}
For each $x \in \OO_D(S)$, the function from $\OO_D(S) \setminus (\cS_x \cup \R)$ to $A$, sending $w$ into $\ck(x,w)$, is summable on $\OO_D(S)$ \wrt $dw$ and the function from $\partial\OO_D(S) \setminus \R$ to $A$, sending $w$ into $\ck(x,w)$, is summable on $\partial\OO_D(S)$ \wrt $d\sigma_w$.
\end{lemma}

Denote by $\overline{\OO}_D$ the closure of $\OO_D$ in $\Q_A$, which coincides with $\OO_{\overline{D}}$ if $\overline{D}$ denotes the closure of $D$ in $\C$.

We are now in position to  state our main result.

\begin{theorem} \label{thm:main}
Let $S$ be a gis of $A$, let $M$ be the vector subspace of $A$ indu\-cing $S$ and let $m:=\dim M$. Choose a volume form $dw$ of $M$ as above and denote by $d\sigma_w$ the corresponding volume form of $\partial\OO_D(S)$.
Then, for each $f \in \mc{S}^1(\overline{\OO}_D,A)$ and for each $x \in \OO_D$, it holds:
\begin{equation} \label{eq:volume-cauchy}
f(x)=\frac{1}{2\pi}\int_{\partial\OO_D(S)}\ck(x,w) \, \nn(w) f(w) \, d\sigma_w-\frac{1}{\pi}\int_{\OO_D(S)}\ck(x,w) \, \dd{f}{x^c}(w) \, dw\;.
\end{equation}
If $f$ is slice regular on $\OO_D$, then formula \eqref{eq:volume-cauchy} holds with only the boundary term.
\end{theorem}

Formula \eqref{eq:volume-cauchy} is still valid for all functions in $\mc{S}^1(\OO_D,A)$, which admit a continuous extension on $\overline{\OO}_D$. This can be seen by approximating the domain with smaller subdomains. Using the same strategy, one can relax also the assumption on the $\mscr{C}^1$-regularity of $\partial D$ and hence of $\partial \OO_D(S)$.

The reader observes that, if $S$ is equal to the gis $\{J,-J\}$ for some $J \in \cS_A$, then formula (\ref{eq:volume-cauchy}) reduces to the Cauchy formula obtained in Theorem~27 of \cite{GhPe_AIM}.  

For each continuous slice function $f \in \mc{S}^0(\partial\OO_D(S),A)$, we can consider the \emph{Cauchy-type integrals} $F_S^+:\OO_D \lra A$ and $F_S^-:\Q_A\setminus\overline{\OO}_D \lra A$ defined respectively by setting
\begin{equation} \label{eq:spj-formula}
F_S^\pm(x):=\frac{1}{2\pi}\int_{\partial\OO_D(S)}\ck(x,w) \, \nn(w) f(w) \, d\sigma_w\;.
\end{equation}
Observe that $F_S^\pm$ are slice regular functions. If the boundary function $f$ is of class $\mscr{C}^1$ (in fact, a H\"older condition suffices), then a \emph{{S}okhotski\u\i--Plemelj jump formula} is valid.

\begin{theorem}\label{thm:jump}
Let $S$ be a gis of $A$, let $f \in \mc{S}^1(\partial\OO_D(S),A)$ and let $F_S^\pm$ be the functions defined in $(\ref{eq:spj-formula})$. Then $F_S^+$ extends continuously to $\overline{\OO}_D$ and $F_S^-$ extends continuously to $\Q_A\setminus\OO_D$. Moreover, for each $x \in \partial\OO_D(S)$, it holds:
\begin{equation} \label{eq:jump}
f(x)=F_S^+(x)-F_S^-(x)\;.
\end{equation}
\end{theorem}

\begin{corollary}\label{cor:ext}
Let $S$ be a gis of $A$ and let $f \in \mc{S}^1(\partial\OO_D(S),A)$. Then there exists $F \in \mc{SR}(\OO_D,A) \cap \mscr{C}^0(\overline{\OO}_D,A)$ such that $F=f$ on $\partial\OO_D(S)$ if and only if $F_S^-$ vanishes on $\partial\OO_D(S)$. In this case, the extension $F$ is given by $F_S^+$. 
\end{corollary}

\begin{remark}
In general, the integrals $F^\pm_S$ depend on $S$. For example, if $A=\HH$, $\OO_D$ is the unit ball, $S=\{i,-i\}$, $S'=\{j,-j\}$ and $f(x)=x_0+ix_1$, then $F^+_S(x)=x$, $F^-_S(x)=0$, while $F^+_{S'}(x)=x/2$, 
$F^-_{S'}(x)=-1/(2x)$. 
\end{remark}

We give the proof of the preceding results in the next section.

We conclude the present section by reformulating our Cauchy formula \eqref{eq:volume-cauchy} in the quaternionic and in the Clifford algebra cases. 


\subsection{The quaternionic case}

If $A$ is the algebra $\HH$ of (real) quaternions, the quadratic cone is the whole algebra (see \cite{GhPe_AIM}) and therefore we can take $M=\HH$ and $S=\cS_\HH$ as a gis. In this case, $\OO_D(S)=\OO_D$ is an open domain in $\HH$, with $\partial\OO_D(S)=\partial\OO_D$.

\begin{corollary}
For each slice function $f \in \mc{S}^1(\overline{\OO}_D,\HH)$ and for each $x \in \OO_D$, it holds:
\begin{equation*} \label{eq:volume-cauchy-H}
f(x)=\frac{1}{2\pi}\int_{\partial\OO_D}C_{\cS_\HH}(x,w) \, \nn(w) f(w) \, d\sigma_w-\frac{1}{\pi}\int_{\OO_D}C_{\cS_\HH}(x,w) \, \dd{f}{x^c}(w) \, dw\;,
\end{equation*}
where, for $x\in\HH$ and $w=w_0+w_1i+w_2j+w_3k\in\HH\setminus\R$ with  $w\notin\cS_x$, the kernel is
\[C_{\cS_\HH}(x,w)=\frac1{2\pi} \, \frac{\Delta_w(x)^{-1}(\overline w-x)}{w_1^2+w_2^2+w_3^2}\;.\]
\end{corollary}


\subsection{The Clifford algebra case}

If $A$ is the real Clifford algebra $\R_n$ with signature $(0,n)$, the quadratic cone $\Q_n$ contains the subspace $\R^{n+1}$ of \emph{paravectors} (see \cite{GhPe_AIM,GhPe_Trends}).
We can then take $M=\R^{n+1}$ and $S=\cS^{n-1}=\{x=x_1e_1+\cdots+ x_ne_n\in\R_n \,|\, x_1^2+\cdots+x_n^2=1\}$ as a gis. Here $e_1,\ldots,e_n$ denote the basic generators of $\R_n$, satisfying the relations $e_ie_j+e_je_i=-2\delta_{ij}$. In this case, $\OO_D(\cS^{n-1})$ is an open domain in $\R^{n+1}$ and we can take $v=(1,e_1,\ldots,e_n)$ as orthonormal basis of $M$.
Slice regularity on $\R_n$ generalizes the concept of \emph{slice monogenic functions} introduced in \cite{CoSaSt2009Israel}. If $D$ intersects the real axis, then the restriction of a slice regular function $f$ on $\OO_D$ to $\R^{n+1}$ is a slice monogenic function. Conversely, each slice monogenic function is the restriction of a unique slice regular function.

\begin{corollary}
Let $w=w_0+w_1e_1+\cdots+w_ne_n \in \R^{n+1} \setminus \R$ and let $\overline w=w_0-w_1e_1-\cdots-w_ne_n$ be its Clifford conjugate. For $x \in \Q_{\R_n}$ and $w \notin \cS_x$, consider the Cauchy kernel 
\[
C_{\cS^{n-1}}(x,w)=\frac2{\eta_{n-1}} \, \frac{\Delta_w(x)^{-1}(\overline w-x)}{(w_1^2+\cdots+w_n^2)^{\frac{n-1}{2}}}\;.
\]
Then, for each slice function $f \in \mc{S}^1(\overline{\OO}_D,\R_n)$ and for each $x \in \OO_D$, it holds:
\begin{equation*} \label{eq:volume-cauchy-Rn}
f(x)=\frac{1}{2\pi}\int_{\partial\OO_D(\cS^{n-1})}C_{\cS^{n-1}}(x,w) \, \nn(w) f(w) \, d\sigma_w-\frac{1}{\pi}\int_{\OO_D(\cS^{n-1})}C_{\cS^{n-1}}(x,w) \, \dd{f}{x^c}(w) \, dw\;.
\end{equation*}
If $f$ is slice monogenic on $\OO_D(\cS^{n-1})$, continuous up to the boundary, and $x\in\OO_D(\cS^{n-1})$,  then
\begin{equation*} \label{eq:monogenic}
f(x)=\frac{1}{2\pi}\int_{\partial\OO_D(\cS^{n-1})}C_{\cS^{n-1}}(x,w) \, \nn(w) f(w) \, d\sigma_w\;.
\end{equation*}
\end{corollary}

Another instance of Cauchy formula \eqref{eq:volume-cauchy} on $\R_n$ can be obtained by choosing the gis $S=\{J,-J\}$ for some imaginary unit $J\in\cS_{\R_n}$. If $J \in \cS^{n-1}\subset\cS_{\R_n}$, then formula (\ref{eq:volume-cauchy}) reduces to the Cauchy and Pompeiu formulas given in \cite{CoSaJMAA2011,CoSaStMMJ2011} for slice monogenic functions.


\section{Proofs}

Let us construct explicitly polar coordinates on the standard sphere $S^n$ of $\R^{n+1}$ and on $\R^{n+1}$ itself.

Let $I_1$ be the interval $(0,2\pi)$ of $\R$, let $I_1^+$ be the interval $(0,\pi)$, let $N$ be the subset $[0,+\infty) \times \{0\}$ of $\R^2=\R \times \R$ and let $\varphi_1:I_1 \lra \R^2 \setminus N$ be the smooth embedding defined by setting $\varphi_1(\theta_1):=(\cos(\theta_1),\sin(\theta_1))^T$. For each $n \geq 2$, identify $\R^n$ with $\R^{n-1} \times \R$ and define the open subsets $I_n$ and $I_n^+$ of $\R^n$ and the smooth embedding $\varphi_n:I_n \lra \R^{n+1} \setminus (N \times \R^{n-2})$ by induction as follows: $I_n:=I_{n-1} \times (-\pi/2,\pi/2)$, $I_n^+:=I_{n-1} \times (0,\pi/2)$ and  $\varphi_n(\theta^{\pr},\theta_n):=(\cos(\theta_n)\varphi_{n-1}(\theta^{\pr}),\sin(\theta_n))^T$.

Let $n \geq 1$. It is easy to verify that $\varphi_n(I_n)$ is equal to the dense open subset $S^n \setminus (N \times \R^{n-1})$ of $S^n$, where $N \times \R^{n-1}$ denotes $N$ if $n=1$. The map $\varphi_n$ is a smooth diffeomorphism onto its image, called polar coordinates on $S^n$. Similarly, if $S^n_+$ denotes the northern hemisphere $S^n \cap \{x_{n+1}>0\}$ of $S^n$, then $\varphi_n$ induces a smooth diffeomorphism from $I_n^+$ to the dense open subset $S^n_+ \setminus (N \times \R^{n-1})$ of $S^n_+$.

Let $\R_*:=\R \setminus \{0\}$. Define $H_{n+1}:=(N \times \R^{n-1}) \cup (\R^n \times \{0\}) \subset \R^{n+1}$ and the polar coordinates $\Phi_{n+1}:\R_* \times I_n^+ \lra \R^{n+1} \setminus H_{n+1}$ of $\R^{n+1}$ by setting $\Phi_{n+1}(\rho,\theta):=\rho\varphi_n(\theta)$. Denote by $J_{\Phi_{n+1}}$ and $J_{\varphi_n}$ the jacobian matrices of $\Phi_{n+1}$ and of $\varphi_n$, respectively. Observe that, given $\rho \in \R_*$ and $\theta \in I_n^+$,  $J_{\Phi_{n+1}}(\rho,\theta)$ is equal to the block matrix $(\varphi_n(\theta)|\rho J_{\varphi_n}(\theta))$.
Define the smooth function $\I_n:I_n^+ \lra \R$ by setting
\begin{equation} \label{eq:I}
\I_n(\theta):=\det \big(J_{\Phi_{n+1}}(1,\theta)\big)=\det \big((\varphi_n(\theta)|J_{\varphi_n}(\theta))\big)
\end{equation}
for each $\theta \in I_n^+$.
Let us prove an elementary, but very useful, lemma.

\begin{lemma} \label{lem:polar}
For each integer $n \geq 1$, the following assertions hold:
\begin{itemize}
 \item[$(\mr{i})$] $\det \big(J_{\Phi_{n+1}}(\rho,\theta)\big)=\rho^n\I_n(\theta)$ for each $(\rho,\theta) \in \R_* \times I_n^+$.
 \item[$(\mr{ii})$] $(\varphi_n)^TJ_{\varphi_n}=0$ on $I_n^+$.
 \item[$(\mr{iii})$] $\I_n(\theta)=\prod_{k=2}^n(\cos(\theta_k))^{k-1}$ for each $\theta=(\theta_1,\ldots,\theta_n) \in I_n^+$. The latter product reduces to $1$ if $n=1$.
 \item[$(\mr{iv})$] $\I_n=\det \big((\varphi_n|J_{\varphi_n})\big)>0$ on $I_n^+$.
 \item[$(\mr{v})$] $\I_n=\sqrt{\det\big((J_{\varphi_n})^TJ_{\varphi_n}\big)}$ on $I_n^+$.
 \item[$(\mr{vi})$] $\int_{I_n^+}\I_n(\theta) \, d\theta=\eta_n/2$.
\end{itemize}
\end{lemma}
\begin{proof}
Since $J_{\Phi_{n+1}}(\rho,\theta)=(\varphi_n(\theta)|\rho J_{\varphi_n}(\theta))$, point $(\mr{i})$ is evident.

Let us prove points $(\mr{ii})$ and $(\mr{iii})$ by induction on $n \geq 1$. If $n=1$, then $(\mr{ii})$ and $(\mr{iii})$ are immediate to verify. Let $n \geq 2$ and let $\theta=(\theta^{\pr},\theta_n) \in I_{n-1} \times (0,\pi/2)=I_n^+$. Define $c_n:=\cos(\theta_n)$, $s_n:=\sin(\theta_n)$, $\varphi_{n-1}^{\pr}:=\varphi_{n-1}(\theta^{\pr})$ and $J_{n-1}^{\pr}:=J_{\varphi_{n-1}}(\theta^{\pr})$. By definition, we have that $\varphi_n(\theta)=(c_n\varphi_{n-1}^{\pr},s_n)^T$. In this way, it holds:
\[
J_{\varphi_n}(\theta)=
\left(
\begin{array}{c|c}
c_nJ_{n-1}^{\pr} & -s_n\varphi_{n-1}^{\pr} \\
\hline
0 & c_n
\end{array}
\right)
\]
and hence
\[
(\varphi_n(\theta))^TJ_{\varphi_n}(\theta)=
\big(
c_n^2(\varphi_{n-1}^{\pr})^TJ_{n-1}^{\pr} \big| -c_ns_n((\varphi_{n-1}^{\pr})^T\varphi_{n-1}^{\pr}-1)
\big).
\]
By induction, we know that $(\varphi_{n-1}^{\pr})^TJ_{n-1}^{\pr}=0$. On the other hand, $\varphi_{n-1}^{\pr}$ belongs to $S^{n-1}$ and hence $(\varphi_{n-1}^{\pr})^T\varphi_{n-1}^{\pr}=1$. It follows that $(\varphi_n(\theta))^TJ_{\varphi_n}(\theta)=0$, as desired. This prove $(\mr{ii})$. Observe that
\begin{align*}
\I_n(\theta)&=
\det\big(J_{\Phi_{n+1}}(1,\theta)\big)=
\det \left(
\begin{array}{c|c|c}
c_n\varphi_{n-1}^{\pr} & c_nJ_{n-1}^{\pr} & -s_n\varphi_{n-1}^{\pr} \\
\hline
s_n & 0 & c_n
\end{array}
\right)= \\
&=
c_n^{n-1}
\det \left(
\begin{array}{c|c|c}
c_n\varphi_{n-1}^{\pr} & J_{n-1}^{\pr} & -s_n\varphi_{n-1}^{\pr} \\
\hline
s_n & 0 & c_n
\end{array}
\right).
\end{align*}
The last determinant can be expanded \wrt the last row, to obtain
\begin{align*}
\I_n(\theta)&=c_n^{n-1}\left((-1)^ns_n
\det\left(\begin{array}{c|c}
 J_{n-1}^{\pr} & -s_n\varphi_{n-1}^{\pr} 
\end{array}\right)+c_n
\det\left(\begin{array}{c|c}
c_n\varphi_{n-1}^{\pr} &  J_{n-1}^{\pr} 
\end{array}\right)
\right)=\\
&=c_n^{n-1}\left(s_n^2
\det\left(\begin{array}{c|c}
\varphi_{n-1}^{\pr} &  J_{n-1}^{\pr} 
\end{array}\right)+c_n^2
\det\left(\begin{array}{c|c}
\varphi_{n-1}^{\pr} &  J_{n-1}^{\pr} 
\end{array}\right)
\right)=c_n^{n-1}\I_{n-1}(\theta')\;.
\end{align*}
Point $(\mr{iii})$ follows by induction. Point $(\mr{iv})$ is an immediate consequence of $(\mr{iii})$ and the fact that the cosine is positive on $(-\pi/2,\pi/2)$.

It remains to prove $(\mr{v})$ and $(\mr{vi})$. Fix an integer $n \geq 1$ and $\theta \in I_n^+$. Bearing in mind $(\mr{ii})$ and the equality $(\varphi_n(\theta))^T\varphi_n(\theta)=1$, it follows that
\[
(J_{\Phi_{n+1}}(1,\theta))^TJ_{\Phi_{n+1}}(1,\theta)=
\left(
\begin{array}{c|c}
1 & 0 \\
\hline
0 & (J_{\varphi_n}(\theta))^TJ_{\varphi_n}(\theta)
\end{array}
\right)
\]
and hence
\[
\I_n(\theta)^2=\det\big((J_{\Phi_{n+1}}(1,\theta))^TJ_{\Phi_{n+1}}(1,\theta)\big)=\det\big((J_{\varphi_n}(\theta))^TJ_{\varphi_n}(\theta)\big)\;.
\]
By combining this fact with $(\mr{iv})$, we infer point $(\mr{v})$. In particular, if $d\xi_n$ is the standard volume form on $S^n$, then $(\varphi_n)^{\ast}(d\xi_n)=\I_n(\theta)d\theta$ 
and hence $(\mr{vi})$ holds:
\[
\int_{I_n^+}\I_n(\theta) \, d\theta=\frac{1}{2}\int_{S^n}d\xi_n=\frac{\eta_n}{2}\;.
\]
The proof is complete.
\end{proof}


\textit{In what follows, we will use the notations fixed in the preceding section}.

\medskip
{\noindent\bf Proof of Lemma~\ref{lem:L^1}:\ }
Identify $\R^{m-1}$ with the vector subspace $\{0\} \times \R^{m-1}$ of $M \simeq \R^m=\R \times \R^{m-1}$ and $S$ with the sphere $S^{m-2}$ in $\R^{m-1}$.

If $m=2$, then Lemma~\ref{lem:L^1} is evident.
Suppose $m \geq 3$ and fix $J \in S^{m-2}$. We will prove the following two inequalities
\begin{equation} \label{eq:L^1-surface}
\int_{\partial\OO_D(S)}\|\ck(x,w)\|d\sigma_w<+\infty
\end{equation}
and
\begin{equation} \label{eq:L^1-volume}
\int_{\OO_D(S)}\|\ck(x,w)\|dw<+\infty\;.
\end{equation}
Evidently, points (\ref{eq:L^1-surface}) and (\ref{eq:L^1-volume}) are equivalent to Lemma~\ref{lem:L^1}.
We organize the remainder of the proof into two steps. In the first, we prove (\ref{eq:L^1-surface}). The second is devoted to the proof of (\ref{eq:L^1-volume}).

\textit{Step I.} 
For simplicity, we assume that $\partial D_J$ is connected. If this is not true, it suffices to consider each connected component of $\partial D_J$, suitably oriented.

Let $u=\alpha_0+\beta_0i \in \partial D$, let $v:=\alpha_0+\beta_0J \in \partial D_J$ 
and let $a,b:(0,1) \lra \R$ be $\mscr{C}^1$-functions such 
that the map $(0,1) \ni t \mapsto a(t)+b(t)J \in \C_J \setminus \{v\}$ is a $\mscr{C}^1$-embedding, which parametrizes $\partial D_J \setminus \{v\}$. Observe that $b(t)$ has at most two zeros in $(0,1)$. Define the $\mscr{C}^1$-map $\Psi:(0,1) \times I_{m-2}^+ \lra \R^m=\R \times \R^{m-1}$ and the dense open subset $\Xi$ of $\partial\OO_D(S)$ by setting
\[
\Psi(t,\theta):=(a(t),b(t)\varphi_{m-2}(\theta))^T
\]
and
\[
\Xi:=\big\{\alpha+\beta K \in \Q_A \, \big| \, \alpha,\beta \in \R, \, \alpha+\beta i \in \partial D \setminus \{u\},  \, K \in S^{m-2}_+ \setminus (N \times \R^{m-3})\big\}.
\]
It is immediate to see that the image of $\Psi$ coincides with $\Xi$ and the restriction $\psi:(0,1) \times I_{m-2}^+ \lra \Xi$ of $\Psi$ onto its image $\Xi$ is a $\mscr{C}^1$-diffeomorphism.

Let $(t,\theta) \in (0,1) \times I_{m-2}^+$ and let $J_{\Psi}(t,\theta)$ be the jacobian matrix of $\Psi$ at $(t,\theta)$. It holds:
\[
J_{\Psi}(t,\theta)=
\left(
\begin{array}{c|c}
a^{\pr}(t) & 0 \\
\hline
b^{\pr}(t)\varphi_{m-2}(\theta) & b(t)J_{\varphi_{m-2}}(\theta)
\end{array}
\right),
\]
where $a^{\pr}$ and $b^{\pr}$ denote the derivatives of $a$ and of $b$, respectively. By points $(\mr{ii})$, $(\mr{iv})$ and $(\mr{v})$ of Lemma \ref{lem:polar}, we obtain:
\[
(J_{\Psi}(t,\theta))^TJ_{\Psi}(t,\theta)=
\left(
\begin{array}{c|c}
(a^{\pr}(t))^2+(b^{\pr}(t))^2 & 0 \\
\hline
0 & (b(t))^2(J_{\varphi_{m-2}}(\theta))^TJ_{\varphi_{m-2}}(\theta)
\end{array}
\right)
\]
and hence
\[
\sqrt{\det\big((J_{\Psi}(t,\theta))^TJ_{\Psi}(t,\theta)\big)}=|b(t)|^{m-2}\I_{m-2}(\theta)\sqrt{(a^{\pr}(t))^2+(b^{\pr}(t))^2}\;.
\]
By combining point $(\mr{iv})$ of Lemma \ref{lem:polar} with the latter equality, we infer that
\begin{equation} \label{eq:psi*}
\psi^{\ast}(d\sigma_w)=|b(t)|^{m-2}\I_{m-2}(\theta)\sqrt{(a^{\pr}(t))^2+(b^{\pr}(t))^2} \, dt \, d\theta.
\end{equation}
Bearing in mind (\ref{eq:psi*}) and performing the change of variable $w=\psi(t,\theta)$, we obtain: 
\begin{align*}
&\int_{\partial\OO_D(S)}\|\ck(x,w)\|d\sigma_w=\frac{2}{\eta_{m-2}}\int_{(0,1) \times I_{m-2}^+}\frac{\|C(x,w)\|}{|b(t)|^{m-2}} \, \psi^{\ast}(d\sigma_w)=\\
&=\frac{2}{\eta_{m-2}}\int_{(0,1) \times I_{m-2}^+}\|C(x,w)\|\I_{m-2}(\theta)\sqrt{(a^{\pr}(t))^2+(b^{\pr}(t))^2} \, dt \, d\theta.
\end{align*}
Since the integrand in the last integral is continuous, positive and bounded on $(0,1) \times I_{m-2}^+$, inequality (\ref{eq:L^1-surface}) holds.

\textit{Step II.} Let us prove (\ref{eq:L^1-volume}). The proof is similar to the one of the preceding step, but slightly simpler. Let $(r,s)$ be the coordinates of $\C \simeq \R^2$. Define the $\mscr{C}^1$-map $\Gamma:(D \setminus \R) \times I_{m-2}^+ \lra \R^m=\R \times \R^{m-1}$ and the dense open subset $\Upsilon$ of $\OO_D(S) \setminus \R$ by setting
\[
\Gamma(r,s,\theta):=(r,\Phi_{m-1}(s,\theta))^T
\]
and
\[
\Upsilon:=\big\{r+sK \in \Q_A \, \big| \, r,s \in \R, \, r+si \in D \setminus \R, \, K \in S^{m-2}_+ \setminus (N \times \R^{m-3})\big\}.
\]
It is easy to verify that the image of $\Gamma$ is equal to $\Upsilon$ and the restriction $\gamma:(D \setminus \R) \times I_{m-2}^+ \lra \Upsilon$ of $\Gamma$ onto its image $\Upsilon$ is a $\mscr{C}^1$-diffeomorphism.

Given $(r,s,\theta) \in (D \setminus \R) \times I_{m-2}^+$, it is immediate to see that the determinant $\det\big(J_{\Gamma}(r,s,\theta)\big)$ of the jacobian matrix of $\Gamma$ at $(r,s,\theta)$ is equal to $\det\big(J_{\Phi_{m-1}}(s,\theta)\big)$ and hence, by point $(\mr{i})$ of Lemma \ref{lem:polar}, we have:
\begin{equation} \label{eq:gamma*}
\left|\det\big(J_{\Gamma}(r,s,\theta)\big)\right|=|s|^{m-2}\I_{m-2}(\theta)\;.
\end{equation}
Using the change of variable $w=\Gamma(r,s,\theta)$ and (\ref{eq:gamma*}), we obtain:
\begin{align*}
\int_{\OO_D(S)}\|\ck(x,w)\|dw &=\frac{2}{\eta_{m-2}}\int_{(D \setminus \R) \times I_{m-2}^+}\frac{\|C(x,w)\|}{|s|^{m-2}} \, |s|^{m-2}\I_{m-2}(\theta) \, dr \, ds \, d\theta=\\
&=\frac{2}{\eta_{m-2}}\int_{D \times I_{m-2}^+}\|C(x,w)\|\I_{m-2}(\theta) \, dr \, ds \, d\theta\;.
\end{align*}
This equality implies immediately (\ref{eq:L^1-volume}).

\medskip
{\noindent\bf Proof of Theorem~\ref{thm:main}:\ }
During this proof, we will use the maps $\psi$ and $\Gamma$ defined in the preceding one. If $m=2$, then, as we have just said, Theorem~\ref{thm:main} is equivalent to Theorem~27 of \cite{GhPe_AIM}. Suppose $m \geq 3$. Let us prove that, fixed any $J \in S$, the following equality holds:
\begin{align} \label{eq:volume}\notag
\int_{\partial\OO_D(S)}\ck(x,w) \, \nn(w) f(w) \, d\sigma_w-2\int_{\OO_D(S)}\ck(x,w) \, \dd{f}{x^c}(w) \, dw
=\\=\int_{\partial D_J}C(x,y)J^{-1} dy \, f(y)-\int_{D_J}C(x,y)J^{-1} dy^c \wedge dy \, \dd{f}{x^c}(y)\;.
\end{align}
Bearing in mind Theorem~27 of \cite{GhPe_AIM}, equation (\ref{eq:volume}) is equivalent to Theorem~\ref{thm:main}.
Observe that, if $(t,\theta) \in (0,1) \times I_{m-2}^+$ and $w=\psi(t,\theta)$, then
\begin{equation} \label{eq:outer-norm}
\nn(w)=\frac{b^{\pr}(t)-a^{\pr}(t)J_\theta}{\sqrt{(a^{\pr}(t))^2+(b^{\pr}(t))^2}}\;,
\end{equation}
where $J_\theta:=\varphi_{m-2}(\theta)\in S$.
Making use of (\ref{eq:psi*}), (\ref{eq:outer-norm}) and of the change of variable $w=\psi(t,\theta)$, we obtain:
\begin{align}\label{eq:surface-int}\notag
&\frac{\eta_{m-2}}{2}\int_{\partial\OO_D(S)}\ck(x,w) \, \nn(w) f(w) \, d\sigma_w=\\\notag
&=\int_{(0,1) \times I_{m-2}^+}C(x,w)(b^{\pr}(t)-a^{\pr}(t)J_\theta)\I_{m-2}(\theta) \, dt \, d\theta \, f(w)=\\
\notag
&=\int_{I_{m-2}^+}\I_{m-2}(\theta) \, d\theta \int_0^1C(x,w)(b^{\pr}(t)-a^{\pr}(t)J_\theta) \, dt \, f(w)=\\
&=\int_{I_{m-2}^+}\I_{m-2}(\theta) \, d\theta \int_0^1C(x,w)\,J_\theta^{-1}d(a(t)+b(t)J_\theta) \, f(w)\;,
\end{align}
since $J_\theta^{-1}=-J_\theta$.
Using the change of variable $w=\Gamma(r,s,\theta)=(r,s\,J_\theta)$ and (\ref{eq:gamma*}) again, we also have:
\begin{align} \label{eq:volume-int}\notag
&\eta_{m-2}\int_{\OO_D(S)}\ck(x,w) \, \dd{f}{x^c}(w) \, dw=
\int_{D \times I_{m-2}^+}C(x,w)\I_{m-2}(\theta) \, 2 \, dr \, ds \, d\theta \, \dd{f}{x^c}(w)=\\
&=\int_{I_{m-2}^+}\I_{m-2}(\theta) \, d\theta \int_{D_{J_\theta}}C(x,y) \, J_\theta^{-1}dy^c \wedge dy \, \dd{f}{x^c}(y)\;,
\end{align}
where $2 \, dr \, ds=J_\theta^{-1}dy^c \wedge dy$ 
if $y=r+sJ_\theta \in D_{J_\theta}$.

From \eqref{eq:surface-int} and \eqref{eq:volume-int}, we get 
\begin{align*}
\frac{\eta_{m-2}}{2}\int_{\partial\OO_D(S)}\ck(x,w) \, \nn(w) f(w) \, d\sigma_w-\eta_{m-2}\int_{\OO_D(S)}\ck(x,w) \, \dd{f}{x^c}(w) \, dw=\\=
\int_{I_{m-2}^+}\I_{m-2}(\theta) \, d\theta \left(\int_0^1C(x,w)J_\theta^{-1}d(a(t)+b(t)J_\theta) \, f(w)-\right.\\
\left.-\int_{D_{J_\theta}}C(x,y) \, J_\theta^{-1}dy^c \wedge dy\, \dd{f}{x^c}(y)\right)\;.
\end{align*}
The Cauchy formula proved in Theorem~27 of \cite{GhPe_AIM} gives that
\begin{align*}
&\int_0^1C(x,w)J_\theta^{-1}d(a(t)+b(t)J_\theta) \, f(w)-
\int_{D_{J_\theta}}C(x,y) \, J_\theta^{-1}dy^c \wedge dy \, \dd{f}{x^c}(y)=\\
&=\int_{\partial D_J}C(x,y)J^{-1} dy \, f(y)-\int_{D_J}C(x,y)J^{-1} dy^c \wedge dy \, \dd{f}{x^c}(y)=2\pi f(x)
\end{align*}
for each $\theta \in I_{m-2}^+$. Therefore, thanks to point $(\mr{vi})$ of Lemma \ref{lem:polar}, we have: 
\begin{align*}
&\frac{\eta_{m-2}}{2}\int_{\partial\OO_D(S)}\ck(x,w) \, \nn(w) f(w) \, d\sigma_w-\eta_{m-2}\int_{\OO_D(S)}\ck(x,w) \, \dd{f}{x^c}(w) \, dw=\\
&=\frac{\eta_{m-2}}{2}\left(\int_{\partial D_J}C(x,y)J^{-1} dy \, f(y)-\int_{D_J}C(x,y)J^{-1} dy^c \wedge dy \, \dd{f}{x^c}(y)\right)=\\ &=\int_{I_{m-2}^+}\I_{m-2}(\theta) \, d\theta \left(2\pi f(x)\right)=\pi\,\eta_{m-2}f(x)\;.
\end{align*}
Equality (\ref{eq:volume}) is proved and the proof is complete.

\medskip
{\noindent\bf Proof of Theorem~\ref{thm:jump}:\ }
Let $f \in \mc{S}^1(\partial\OO_D(S),A)$.
For each $\theta\in I^+_{m-2}$, let $J_\theta=\varphi_{m-2}(\theta)\in S$ be as in the preceding proof. Denote by $F_\theta^+$ and $F_\theta^-$ the Cauchy-type integrals defined on $D_{J_\theta}$ and $\C_{J_\theta}\setminus\overline{D}_{J_\theta}$ by
\be\label{Ftheta}
F_\theta^\pm(x):=\frac{1}{2\pi}\int_{\partial{D}_{J_\theta}}C(x,w) J_\theta^{-1} \, dw \,f(w).
\ee
Since the restriction $f_{|\partial D_{J_\theta}}$ is of class $\mscr{C}^1$, $F_\theta^+$ extends continuously to $\overline{D}_{J_\theta}$ and  $F_\theta^-$ extends as a continuous function on  $\C_{J_\theta}\setminus{D}_{J_\theta}$. Moreover, the classical Sokhotski\u\i-Plemelj formula (see \cite[\S2]{Kytmanov}) holds:
\be\label{jump}
F_\theta^+(x)-F_\theta^-(x)=f(x)\quad\text{for each $x\in\partial{D}_{J_\theta}$.}
\ee
The functions $F_\theta^+$ and $F_\theta^-$ are slice regular on $\OO_D$ and $\Q_A \setminus \overline{\OO}_D$, respectively. The continuity of their restrictions to $\C_{J_\theta}$ up to the boundary implies the continuity of the inducing stem functions (see the proof of Proposition 5 of \cite{GhPe_AIM}). In view of Proposition 7(1) of \cite{GhPe_AIM}, also the functions $F_\theta^\pm$ are continuous up to the boundary. Given a point $\hat{x}=\alpha+\beta I\in\partial{\OO_D(S)}$, let $x'=\alpha+\beta J_\theta$ and  $x''=\alpha-\beta J_\theta$. From \eqref{jump}, we get: $f(x')=F_\theta^+(x')-F_\theta^-(x')$ and $f(x'')=F_\theta^+(x'')-F_\theta^-(x'')$. The representation formula (see \cite[Prop.~6]{GhPe_AIM}) applied to $f$ and to $F_\theta^\pm$ gives
\be\label{jump1}
f(\hat{x})=\frac12(f(x')+f(x''))-\frac{IJ_\theta}2(f(x')-f(x''))=F_\theta^+(\hat{x})-F_\theta^-(\hat{x})\;.
\ee
From \eqref{eq:surface-int} and \eqref{Ftheta}, we get:
\begin{align}\label{int_Ftheta}
F_S^\pm(x)&=\frac{1}{2\pi}\int_{\partial\OO_D(S)}\ck(x,w) \, \nn(w) f(w) \, d\sigma_w=\notag\\
&=\frac1{\pi\eta_{m-2}}\int_{I_{m-2}^+}\I_{m-2}(\theta) \, d\theta \int_0^1C(x,w)\,J_\theta^{-1}d(a(t)+b(t)J_\theta) \, f(w)
=\notag\\
&=\frac2{\eta_{m-2}}\int_{I_{m-2}^+}\I_{m-2}(\theta) \, d\theta\, F_\theta^\pm(x)\;.
\end{align}
Let $U$ be a compact  neighborhood of $\partial D$, invariant w.r.t.\ complex conjugation. There exists a positive constant $K$ such that $\|F_\theta^+(x)\|\le K$ for each $x\in \overline{D}_{J_\theta}$ and $\|F_\theta^-(x)\|\le K$ for each $x\in (\C_{J_\theta}\setminus D_{J_\theta})\cap\OO_U$. The representation formula gives $\|F_\theta^+(x)\|\le 2K$ for each $x\in\overline{\OO}_D$ and $\|F_\theta^-(x)\|\le 2K$ for each $x\in (\Q_A\setminus\OO_D)\cap\OO_U$.
Then Lebesgue's dominated convergence theorem ensures that  $F_S^\pm$ extend continuously up to the boundary $\partial{\OO_D}$. Moreover, from \eqref{jump1} and \eqref{int_Ftheta}, we get the jump $F_S^+(\hat{x})-F_S^-(\hat{x})$ at $\hat{x} \in \partial\OO_D(S)$:
\begin{align}
F_S^+(\hat{x})-F_S^-(\hat{x})&=\frac2{\eta_{m-2}}\int_{I_{m-2}^+}\I_{m-2}(\theta) \, d\theta\, (F_\theta^+(\hat{x})-F_\theta^-(\hat{x}))=\notag
\\
&=\frac2{\eta_{m-2}}\int_{I_{m-2}^+}\I_{m-2}(\theta) \, d\theta\, f(\hat{x})=f(\hat{x})\;.
\end{align}




\begin{thebibliography}{10}
\providecommand{\natexlab}[1]{#1}

\bibitem[1]{GeSt2006CR}
G. Gentili and D.C. Struppa, {\itshape A new approach to {C}ullen--regular
  functions of a quaternionic variable}, C. R. Math. Acad. Sci. Paris 342 (2006), pp. 741--744.

\bibitem[2]{GeSt2007Adv}
---{}---{}---, {\itshape A new theory of regular functions of a quaternionic variable}, Adv. Math. 216 (2007), pp. 279--301.

\bibitem[3]{CoSaSt2009Israel}
F. Colombo, I. Sabadini, and D.C. Struppa, {\itshape Slice monogenic functions}, Israel J. Math. 171 (2009), pp. 385--403.

\bibitem[4]{GhPe_AIM}
R. Ghiloni and A. Perotti, {\itshape Slice regular functions on real
  alternative algebras}, Adv. Math. 226 (2011), pp. 1662--1691.

\bibitem[5]{GhPe_Trends}
---{}---{}---, {\itshape A new approach to slice regularity on real algebras},
   in {\itshape Hypercomplex analysis and its Applications},   Birkh\"auser,
  Basel, 2011, pp. 109--124.

\bibitem[6]{CoGeSaPreprint2008}
F. Colombo, G. Gentili, and I. Sabadini, {\itshape A {C}auchy kernel for slice regular functions}, Ann.\ Global Anal.\ Geom. 37 (2010), pp. 361--378.

\bibitem[7]{CoSaJMAA2011}
F. Colombo and I. Sabadini, {\itshape The {C}auchy formula with s--monogenic kernel and a functional calculus for noncommuting operators}, J. Math. Anal. Appl. 373 (2011), pp. 655--679.

\bibitem[8]{CoSaTM2009}
---{}---{}---, {\itshape A structure formula for slice monogenic functions and
  some of its consequences},  in {\itshape Hypercomplex analysis},
  Birkh\"auser, Basel, 2009, pp. 101--114.

\bibitem[9]{CoSaStMMJ2011}
F. Colombo, I. Sabadini, and D.C. Struppa, {\itshape The {P}ompeiu formula for
  slice hyperholomorphic functions}, Michigan Math. J. 60 (2011), pp. 163--170.

\bibitem[10]{Kytmanov}
A.M. Kytmanov, {\itshape The {B}ochner-{M}artinelli integral and its
  applications},    Birkh\"auser Verlag, Basel, 1995.

\end{thebibliography}


\end{document}